\documentclass[reqno,12pt]{amsart}
\usepackage{bbm}
\usepackage[all,pdf]{xy}
\usepackage{epsfig}
\usepackage{amsmath}
\usepackage{amssymb}
\usepackage{amscd}
\usepackage{graphicx}
\usepackage{enumerate}
\usepackage{comment}
\usepackage[colorlinks=true]{hyperref}
\hypersetup{urlcolor=red, citecolor=blue}
\makeatletter
\@namedef{subjclassname@2020}{%
	\textup{2020} Mathematics Subject Classification}

\newtheorem{thm}{Theorem}[section]
\newtheorem{lem}[thm]{Lemma}

\newtheorem{cl}[thm]{Claim}
\newtheorem{prop}[thm]{Proposition}

\newtheorem{cor}[thm]{Corollary}

\newtheorem{rem}[thm]{Remark}

\def \N {\mathbb N}

\def \Z {\mathbb Z}

\def \M {\mathcal M}

\def\supp {\mathrm{supp}}

\def\Pu {\mathcal{ P}^\mu_X}
\def\Bm {\mathcal{B}_X^\mu}
\def\mxg {\mathcal{M}(X,G)}
\numberwithin{equation}{section}

\begin{document}
	\title[Independence, sequence entropy and mean
	sensitivity]{Independence, sequence entropy and mean
		sensitivity for invariant measures}

	\author{Chunlin Liu}
	\address[C. Liu]{School of Mathematical Sciences, Dalian University of Technology, Dalian, 116024, P.R. China}
	\email{chunlinliu@mail.ustc.edu.cn}	
	
	\author{Leiye Xu}
	\address[L. Xu]{School of Mathematical Sciences, University of Science and Technology of China, Hefei, Anhui, 230026, P.R. China}
	\email{leoasa@mail.ustc.edu.cn}
	
	\author{Shuhao Zhang}
	\address[S. Zhang]{School of Mathematical Sciences, University of Science and Technology of China, Hefei, Anhui, 230026, P.R. China}
	\email{yichen12@mail.ustc.edu.cn}

	\subjclass[2020]{37A15, 37A35, 37B05}

	\keywords{IT tuple, sequence entropy tuple, mean sensitive tuple}
	
	%\thanks{}
	\begin{abstract}We investigate the connections between independence, sequence entropy, and mean sensitivity for a measure preserving system under the action of a countable infinite discrete group. We establish that every sequence entropy tuple for an invariant measure is an IT tuple. 
		Furthermore, if the acting group is amenable, we show that for an ergodic measure, the sequence entropy tuples, the mean sensitive tuples along some  tempered F{\o}lner sequence, and the sensitive in the mean tuples along  some  tempered F{\o}lner sequence coincide. 
	\end{abstract}
	
	\maketitle
	\section{Introduction}
	
	Throughout this paper, 	 unless otherwise stated, 
	we assume that  $G$ is a countable infinite discrete group and $X$ is a compact metric
	space with a metric $d$ and $\mathcal{B}_X$ is the Borel $\sigma$-algebra on $X$. By a \emph{topological dynamical system} (tds for short)  we mean a pair $(X,G)$, where $G$ acts on $X$ as a group of homeomorphisms. Denote by $\mathcal{M}(X,G)$ and $\mathcal{M}^e(X,G)$ the set of  $G$-invariant  probability measures and the set of ergodic invariant measures on $(X,\mathcal{B}_X)$, respectively.  By a \emph{measure preserving system} (mps for short)  we mean a triple $(X,\mu,G)$, where $(X,G)$ is a tds and $\mu\in\mxg$.

	Given a dynamical system, one might not only be interested in its overall complexity but also in identifying the specific points in the phase space that contribute to its complexity. For the case $G=\mathbb{Z}$, Blanchard \cite{B1} introduced the concept of entropy pairs to address this question, using them to locate where the entropy lies. He showed that a tds has positive entropy if and only if it has essential entropy pairs. Building on this foundation, Blanchard et al. \cite{B2} extended the notion to define entropy pairs for an invariant probability measure. Later, Glasner and Weiss \cite{GW} generalized this to entropy tuples, while Huang and Ye \cite{HY} further developed entropy tuples in the context of invariant measures.  This approach, now known as local entropy theory  has become a cornerstone in the study of entropy theory. 
	
	As in classical local entropy theory, sequence entropy, originally introduced by Ku\v{s}nirenko \cite{K1967_entropy}, can also be studied in a localized manner. In \cite{HLSY1, Huang_Sequenceentropypairs}, Huang, Li, Maass, Shao and Ye explored sequence entropy tuples in both measure-theoretic and topological contexts, providing foundational results and applications. For a broader introduction of local entropy theory, we refer the reader to the surveys by Glasner and Ye \cite{GY} and Garc\'{\i}a-Ramos and Li \cite{GH} for a comprehensive overview.
	
	Kerr and Li \cite{KH, KH09} extended this framework by employing combinatorial techniques to investigate  entropy tuples and sequence entropy tuples via the concept of independence sets. Notably, both entropy pairs and sequence entropy pairs have been extensively studied from both measure-theoretic and topological perspectives. 
	Meanwhile, they also introduced IT tuples to study tameness, originally defined by Köhler \cite{K} (see an equivalent definition by Glasner \cite{G3}), and proved that every IT tuple is a topological sequence entropy tuple. It is well-established that  each measure-theoretic sequence entropy tuple is also a topological sequence entropy tuple \cite[Corollary 3.6]{Huang_Sequenceentropypairs}. This naturally raises the question: what is the relationship between measure-theoretic sequence entropy tuples and IT tuples?
	
	In this paper, we answer this question by establishing that every sequence entropy tuple for an invariant measure is an IT tuple. More concretely,
	for a tds $(X,G)$ and a $K$-tuple $(A_1,\ldots, A_K)$ of subsets of $X$,
	we say that a set $J\subset G$ is an independence set for $(A_1,\ldots, A_K)$ if for every nonempty finite subset
	$I\subset J$ and function $\sigma : I\to  \{1, 2,\ldots, K\}$ we have
	$$\cap_{g\in I}g^{-1}(A_{\sigma(g)})\neq\emptyset.$$
	
		Given $K\ge2$, a tuple $ (x_k)_{k=1}^K\in X^{(K)}$ is said to be an {\it IT $K$-tuple (or IT pair for $K=2$)}  if for any open
		neighborhood $U_k$ of $x_k$, $k=1,2,\ldots,K$, there exists an infinite  independence set for $(U_1,\ldots,U_K)$. Denote by $IT_K(X,G)$ the set of IT $K$-tuples.

	Now we state our first result in this paper.
	\begin{thm}\label{main_abel}
		Let $(X, G)$ be a tds, and  $\mu \in \mathcal{M}(X, G)$. Then \begin{equation}\label{eq:2012}
			SE_K^\mu(X,G)\subset IT_K(X,G),
		\end{equation}
		where $SE_K^\mu(X,G)$ is the set of sequence entropy $K$-tuples for the measure $\mu$. Thus, combining this with the result in \cite{KH}, one has 
		\[\overline{\cup_{\mu\in\M(X,G)}SE_K^\mu(X,G)}\subset IT_K(X,G)\subset SE_K(X,G),\]
		where $SE_K(X,G)$ is the set of topological sequence entropy $K$-tuples.
	\end{thm}
	\begin{rem}
		The inclusion relation in \eqref{eq:2012} can be strict. For instance, consider a $\Z$-tds that is minimal, topologically (strongly) mixing, and admits an ergodic measure $\mu$ with  discrete spectrum. Such a system does not have any $\mu$-sequence entropy tuple, yet it possesses arbitrarily large IT tuples (see e.g. \cite[Theorem 8.6]{KH}). The existence of this class of systems is supported by the fact that every ergodic mps has a strictly ergodic model that is topologically mixing \cite{Lehrer1987}.
	\end{rem}
	
Building on the results previously obtained via the localization approach, we now turn to deriving results on global complexity. Given $K\ge 2$, we call an $K$-tuple $\left(x_1, x_2, \ldots, x_K\right) \in X^{(K)}$ \textit{essential} if $x_i \neq x_j$ for all $1 \leq i < j \leq K$.  
 Huang and Ye \cite[Theorem 3.8]{Huang_Ye2009}  established that for $G = \mathbb{Z}$ and $\mu \in \mathcal{M}(X, G)$, the maximal $\mu$-sequence entropy  $h_\mu^*(G) > \log(K-1)$ implies $(X,G)$ has essential sequence entropy $K$-tuples for $\mu$. We now generalize this result to the setting of countable infinite discrete groups.
	\begin{thm}\label{thm-sequence-entropy}	Let $(X, G)$ be a tds and $\mu\in\mathcal{M}(X,G)$. Let $K\in\N$ with $K\ge 2$. If  $(X,G)$ has  no essential sequence entropy $K$-tuples for $\mu$, then $h_\mu^*(G)\le \log(K-1)$.
	\end{thm}
	\begin{rem}
The converse of this theorem may not hold. That is, for $ \mu \in \mathcal{M}(X, G) $, even if $(X, G)$ admits essential sequence entropy $K$-tuples for $\mu$, it is possible that $h_\mu^*(G) < \log K$.

For instance, let $\mu_1, \mu_2 \in \mathcal{M}(X, G)$ with $h_{\mu_1}^*(G) = \log 2$ and $h_{\mu_2}^*(G) = \log 3$, and consider the convex combination $\mu = \frac{1}{2} \mu_1 + \frac{1}{2} \mu_2$. Then it is easy to verify that  
\[
h_\mu^*(G) = \frac{1}{2} \log 2 + \frac{1}{2} \log 3 < \log 3, \quad \text{while} \quad SE_3^\mu(X, G) = SE_3^{\mu_2}(X, G) \neq \emptyset.
\]

However, if $\mu$ is ergodic, then the converse does hold (see \cite[Theorem 3.8]{Huang_Ye2009} for the case $G = \mathbb{Z}$ and \cite[Theorem 2.9]{LiuWangxu2024} for general $G$).
	\end{rem}
	
	  Huang \cite{H1}  showed that a tame tds has zero maximal sequence entropy for any invariant measure.
 Kerr and Li \cite{KH} later characterized tameness by showing that a tds is tame if and only if it admits no essential IT pairs.
 These two results naturally lead to the following question: Does the absence of essential IT tuples impose an upper bound on the measure-theoretic sequence entropy?
 By combining Theorems \ref{main_abel} and \ref{thm-sequence-entropy}, we provide an affirmative answer by extending Huang's result to the  multiple case.
	\begin{cor}\label{cor-entropy-bound}
		Let $(X, G)$ be a tds. If $X$ admits no essential IT $K$-tuples, then 	for every $\mu \in \mathcal{M}(X, G)$,
		$$
		h_\mu^*(G) \leq \log (K-1).
		$$
	\end{cor}

	In addition to the  entropy,  the sensitivity is another candidate to describe the complexity. Huang, Lu and Ye \cite{HLY} studied sensitivity and its local representation from measure-theoretic viewpoint. The notion of mean sensitivity  was investigated by Garc\'{\i}a-Ramos \cite{G17}. Li \cite{L16} introduced the concept of mean $n$-sensitivity and showed that an ergodic mps is mean $n$-sensitive if and only if its maximal sequence entropy with respect to this ergodic measure is not less than $\log n$. Li, Ye and Yu \cite{LYY22} introduced the notion of   $n$-sensitivity in the mean, which was proven to be equivalent to mean $n$-sensitivity in the ergodic case.
	
Using the idea of localization, García-Ramos \cite{G17} introduced the notion of mean sensitive pairs for a measure. Li and Yu \cite{LY21} extended this concept to tuples and showed that every entropy tuple of an ergodic measure-preserving system is a mean sensitive tuple. However, the authors of this paper left an open question about whether sequence entropy tuples in a mps can be characterized using mean sensitivity tuples.  This question was recently resolved by García-Ramos and Muñoz-López \cite{GM22} for pairs in ergodic systems with abelian group actions (a similar result for tuples was established in \cite{LLTS2024}). For non ergodic case, an counterexample appears in \cite[Theorem 1.6]{LLTS2024}. 
	However, previous methods are restricted to abelian groups, leaving the problem unresolved for non-abelian groups. In this paper, we provide a positive answer for ergodic systems with amenable group actions.
	
	\begin{thm}\label{cor:se=sm}Let $(X, G)$ be a tds, where $G$  is amenable, and let $\mu\in \mathcal{M}^e(X, G )$. For any tempered F{\o}lner sequence $\mathcal{F}$   of $G$ and $K\in\N$, the $\mu$-sequence entropy $K$-tuple, the $\mu$-mean sensitive $K$-tuple  along $\mathcal{F}$ and the $\mu$-sensitive in the mean $K$-tuple along $\mathcal{F}$ coincide.
	\end{thm}
	
	{\bf Structure of this paper:} In Section \ref{sec-pre}, we introduce fundamental concepts and preliminary results. Theorem \ref{main_abel} is proved in Section \ref{sec-it}, and Theorem \ref{thm-sequence-entropy} is established in Section \ref{sec:new}. The proof of Theorem \ref{cor:se=sm} is given in Section \ref{sec:proof_thm_se_ms}.

	%\subsection{Haar Measures on Homogeneous Spaces}
	%A compact topological group $Z$ has a unique nontrivial \emph{Haar measure} $\nu$ that is nontrivial   Baire probability measure 
	%and invariant under all left rotations,  right rotations and  inversion. Refer {\cite[Theorem 5.18]{Eisner2015}},
	%    for a closed subgroup $H $  of $Z $,  there is a unique Baire probability measure $v^* $  on $Z / H $  that is invariant under all rotations by elements of $Z $.

	\section{Preliminaries}\label{sec-pre}

	\subsection{F\o lner sequence}
	
	A countable group $G $  is \emph{amenable} \cite{F} if it possesses a \emph{F{\o}lner sequence}, which is a sequence of finite nonempty subsets $\mathcal{F} = \{F_n\}_{n=1}^{\infty} $  such that
	\[
	\lim_{n \to +\infty} \frac{|gF_n \Delta F_n|}{|F_n|} = 0 \quad \text{for every } g \in G,
	\]
	where $|A|$ is the number of elements in the set $A$.
	A F{\o}lner sequence $\mathcal{F}= \{F_n\}_{n=1}^{\infty} $  is said to be \emph{tempered} if there exists a constant $C > 0 $  such that
	\[
	|\cup_{k < n} F_k^{-1}F_n| \leq C |F_n| \quad \text{for all } n \in \mathbb{N}.
	\]
	From Lindenstrauss \cite{Lindenstrauss_Pointwise_amenable}, every F{\o}lner sequence has a tempered F{\o}lner subsequence, and the following pointwise ergodic theorem holds.
	\begin{thm}\label{tempered_pointwise_ergodic}
		Let $(X,\mu,G)$ be an mps and $\mathcal{F}=\{F_n \}_{n=1}^{\infty} $ be a tempered F{\o}lner sequence of $G$. Then for any $f \in L^1(X,\mu)$, there is a $G$-invariant $f^* \in L^1(X,\mu)$ such that 
		\begin{equation*}
			\lim_{n \to +\infty} \frac 1{|F_n|} \sum_{g\in F_n} f (gx) = f^* (x) \ \text{ for } \mu\text{-a.e. }x\in X.
		\end{equation*} 
		In particular, if $\mu$ is ergodic, one has
		\begin{equation*}
			\lim_{n \to +\infty} \frac 1{|F_n|} \sum_{g\in F_n} f (gx) = \int f d\mu  \ \text{ for } \mu\text{-a.e. }x\in X.
		\end{equation*} 
	\end{thm}

	\subsection{Factor map}
	Let $(X,G)$ and $(Y,G)$ be two tdss, and $\mu\in\mxg$. In this paper, we use the following  two types of factor maps:
	\begin{itemize}
		\item a \emph{topological factor map} $\pi:(X,G)\to (Y,G)$ is a continuous surjective map such that $g\circ\pi=\pi\circ g$ for any $g\in G$;
		\item a \emph{measure-theoretic factor map} $\pi:(X,\mu,G)\to (Y,\pi_*(\mu),G)$ is a measurable map such that for any $g\in G$, $g\circ\pi=\pi\circ g$, $\mu$-a.e., where  $\pi_*(\mu):=\mu\circ\pi^{-1}$.
	\end{itemize}	
	If no ambiguity arises in the following text, we will omit ``topological" or ``measure-theoretic" and refer to them simply as factor.
	
	\subsection{Metric on the space of probability measures}
	Let $ X $ be a compact metric space with Borel $\sigma$-algebra $ \mathcal{B}_X $. Let $ \mathcal{M}(X) $ be the space of Borel probability measures on $ X $. For $ \mu \in \mathcal{M}(X) $, denote by $ \mathcal{B}_X^\mu $ the completion of $ \mathcal{B}_X $ with respect to $ \mu $, i.e., the smallest $\sigma$-algebra containing $ \mathcal{B}_X $ and all $\mu$-null sets.
	If $\{f_n \}_{n=1}^\infty$ is a dense subset of $C(X)$, the set of continuous functions on $X$, then  
	\begin{equation}\label{metric_measure}
		\varrho(\mu, \nu) = \sum_{n=1}^\infty \frac{|\int f_nd\mu-\int f_nd\nu|}{2^n\|f_n\|_\infty},
	\end{equation}
	is a metric on $\mathcal{M}(X)$ giving the weak$^*$ topology, i.e., 
	\begin{equation}\label{eq123}
	\mu_k \stackrel{\varrho}{\to} \mu \quad \iff \quad \int_X f \, d\mu_k \to \int_X f \, d\mu \quad \forall f \in C(X).
\end{equation}
	Equipped with the weak$^*$ topology, $ \mathcal{M}(X) $ is a compact convex space.	 In the following, we will denote by $\supp(\rho)$ the topological support of $\rho\in\M(X)$.

		\begin{lem}\label{lem-1}Let $\rho\in\mathcal{M}(X)$, $x\in\supp(\rho)$, and let $U$ be an open neighborhood of $x$ in $X$. Then, there exists $\delta>0$ such that if $\rho'\in\mathcal{M}(X)$ satisfying $\varrho(\rho,\rho')<\delta$ then $\rho'(U)>0$. 
	\end{lem}
	\begin{proof}
		Let $ x \in \operatorname{supp}(\rho) $. Then for any open neighborhood $U$ of $x$ in $X$,  $ \rho(U) > 0 $. Due to Urysohn's Lemma, there exists $f\in C(X)$ such that $\int fd\mu>\epsilon$ for some $\epsilon>0$, and $f|_{X\setminus U}=0$. By \eqref{eq123}, there exists $ \delta > 0 $ such that if $ \varrho(\rho, \rho') < \delta $, then 
		$$
		\int fd \rho'-\int fd\rho>\epsilon/2.
		$$
	Thus, $\int fd\rho'>0$. Since    $f|_{X\setminus U}=0$, it follows that $\rho'(U)>0$. 
	\end{proof}
	
	\subsection{Disintegration of probability measures}
	Let $X$ be a compact metric space and $\mu \in \mathcal{M}(X)$.
	For a sub-$\sigma$-algebra  $\mathcal{C}$ of  $\mathcal{B}_X^\mu$,
	$\mu$ can be disintegrated over $\mathcal{C}$ as $\mu = \int_X
	\mu_{\mathcal{C},x}d\mu(x)$ where $\mu_{\mathcal{C},x} \in \mathcal{M}(X)$ such that  for any $f \in L^1(X, \mu)$, $\int_Xfd\mu_{\mathcal{C},x}$ is $\mathcal{C}$-measurable.  Moreover, 
	\begin{align*}\mathbb{E}_\mu( f|\mathcal{C})(x)= \int_Xfd\mu_{\mathcal{C},x} \text{ for } \mu\text{-a.e. }x \in X,
	\end{align*}
	where $ \mathbb{E}_\mu( f|\mathcal{C})$ is the conditional expectation of $f$ over the $\sigma$-algebra $\mathcal{C}$ with respect to the measure $\mu$. 	Especially, for $ \mu\text{-a.e. }x \in X$, 
	\begin{align}\label{eq1-1}	\mu_{\mathcal{C},x}=	\mu_{\mathcal{C},x'}\text{ for $\mu_{\mathcal{C},x}$ a.e. }x'\in X.	\end{align}
	It is well-known that if $\mathcal{C}',\mathcal{C}$ are sub-$\sigma$-algebras of $\mathcal{B}_X^\mu$ with $\mathcal{C}'\subset \mathcal{C}$ then for every $f \in L^1(X,\mathcal{B}_X^\mu,\mu)$, $$\mathbb{E}_\mu( f|\mathcal{C}')(x)=\mathbb{E}_\mu(\mathbb{E}_\mu( f|\mathcal{C})|\mathcal{C}')(x) \text{ for } \mu\text{-a.e. }x \in X.$$
	Thus, $\text{ for all }f\in L^1(X, \mu)$,
	\begin{align*}
		\int_Xfd\mu_{\mathcal{C}',x}= \int_X \int_Xf d\mu_{\mathcal{C},y} d\mu_{\mathcal{C}',x}(y)\text{ for } \mu\text{-a.e. }x \in X.
	\end{align*}
	This implies
	\begin{align}\label{eq1}
		\mu_{\mathcal{C}',x}=  \int_X\mu_{\mathcal{C},y} d\mu_{\mathcal{C}',x}(y)\text{ for } \mu\text{-a.e. }x \in X.
	\end{align}
	
	For a  measure-theoretic factor map $\pi:(X,\mu,G)\to (Y,\pi_*(\mu),G)$, there is a natural disintegration of $\mu=\int_{Y}\mu_yd\nu(y)$, named by the disintegration of $\mu$ over $\pi$, such that 
	$$\mu_{\pi(x)}=\pi_*(\mu_{\pi^{-1}(\mathcal{B}_Y^{\nu}),x})\text{ for } \mu\text{-a.e. }x \in X.$$ We will proceed according to which of the two representations of measure disintegration is more convenient and flexible to work with.

Ergodic decomposition is a standard approach to disintegrating invariant measures. Specifically,  using Choquet representation theorem, for each $\mu \in \mathcal{M}(X,G)$ there is a unique measure $\tau$ on the Borel subsets of the  compact space $\mathcal{M}(X,G)$ such that $\tau (\mathcal{M}^e(X,G)) =1$ and $\mu = \int_{\mathcal{M}^e(X,G)} m \mathrm{d} \tau (m)$, which is called the \emph{ergodic decomposition} of $\mu$.
If we define $$\mathcal{I}_\mu(G)=\{B\in\mathcal{B}_X^\mu:\mu(gB\Delta B)=0,\forall g\in G\},$$
	he ergodic decomposition of $\mu$ is  precisely  the  disintegration of $\mu$ over $\mathcal{I}_\mu(G)$.
	
	\subsection{Kronecker  factor}
	Let $(X,\mu,G)$ be  an mps. Denote $\mathcal{H}= L^2(X,\mathcal{B}_X^\mu,\mu)$, and define
	\[
	\mathcal{H}_c = \left\{\varphi \in \mathcal{H} : \overline{\{U_g \varphi : g \in G\}} \text{ is compact in } \mathcal{H}\right\},
	\]
	where $U_g \varphi(x) = \varphi(gx)$ for $g \in G$ and $x\in X$. According to \cite[Theorem 7.1]{Zim76}, there exists a $G $-invariant sub-$ \sigma $-algebra $\mathcal{K}_\mu(G) $  of $\mathcal{B}^\mu_X $, which is called the \emph{Kronecker algebra} of $(X,\mu,G)$, such that $\mathcal{H}_c = L^2(X, \mathcal{K}_\mu(G), \mu) $. If  $ \pi: (X, \mu, G) \to (Y, \nu, G) $  is a factor map with  $ \pi^{-1}(\mathcal{B}_Y^{\nu}) = \mathcal{K}_\mu(G) $, then  $ (Y, \nu, G) $  is referred to as the \emph{Kronecker factor} of  $ (X, \mu, G)$. 
	The following lemma, adapted from \cite[Theorem 9.21]{Glasner_Eliv2003}, provides the structure of the Kronecker algebra for product systems.
	\begin{lem}\label{n_kronecker}Let   $(X, \mu, G)$ be an mps. Then, $\mathcal{K}_{\mu^{(n)}}(G)= (\mathcal{K}_{\mu}(G))^{(n)}$ for $n\in\N$.
	\end{lem}
	An mps $(X,\mu, G)$ is said to have  \emph{discrete spectrum} if $\mathcal{K}_\mu(G) = \mathcal{B}^\mu_X $. It is well known that if $G$ is abelian, then an ergodic mps $(X,\mu,G)$ with discrete spectrum is isomorphic to a rotation on a compact abelian group \cite{Halmos_von_Neumann1942}. In the case where $G$ is non-abelian, the structure becomes somewhat more intricate. However, we still have the following characterization  \cite[Theorem 1]{Mackey_George1964}.
	\begin{prop}\label{vonNeumann}
		An ergodic mps $(Y,\nu, G)$  with discrete spectrum  is isomorphic to an mps  $(Z/H, \nu_{Z/H}, G)$, where
		\begin{itemize}
			\item[(1)]$Z$ is a compact metrizable topological group, and $H$ is a closed subgroup of $Z$;
			\item[(2)] $G$ acts on $Z/H$ via a continuous group homomorphism $\theta: G \to Z$, with a dense image $\Gamma:= \theta(G)$ and quotient map $\tau: Z \to Z/H$;
			\item[(3)] $\nu_{Z/H}$ is the Haar measure on $Z/H$ induced by the quotient map $\tau$, such that $\nu_{Z/H} = \nu_Z \circ \tau^{-1}$, where $\nu_Z$ is the Haar measure on $Z$.
		\end{itemize}
	\end{prop}
	The following lemma plays an important role in the proof of
	% Theorem \ref{main_abel} and
	Theorem
	\ref{cor:se=sm}, which can be proved by the invariance of the Haar measure.
	\begin{lem}\label{ll}
		Let $Z$ be a compact metrizable topological group, and $\nu_Z$ be the  Haar measure of $Z$. Assume that $Z_1,\ldots,Z_M$ are measurable subsets of $Z$ with $$\nu_Z(\cap_{m=1}^M Z_m)>\kappa.$$ Then there exists an open neighborhood $V$ of $e_Z$ such that	for any  $z_1,\ldots,z_M\in V$.
		$$\nu_Z(\cap_{m=1}^M(Z_mz_m^{-1}))>\kappa,$$
		where $e_Z$ is the identity of $Z$.
	\end{lem}
	
	\subsection{Sequence entropy}Let $X$ be a compact metric space. A partition of $X$ is a collection of subsets of $X$ such that the subsets are pairwise disjoint and their union has full measure. Denote by $\mathcal{P}_X$ the set of all finite  partitions of $X$.
	For $\mu\in\mathcal{M}(X)$, define
	$$\mathcal{P}_X^\mu=\{
	\alpha\in \mathcal{P}_X : \text{each element in }\alpha \text{ belongs to }\mathcal{B}_X^\mu
	\}.$$
	Given $\alpha\in\mathcal{P}_X^\mu$
	and a sub-$\sigma$-algebra $\mathcal{D}$ of $\mathcal{B}_X^\mu$, define
	$$H_\mu(\alpha|\mathcal{D}) = \sum_{A\in\alpha}\int_X-\mathbb{E}_\mu(1_A|\mathcal{D})\log \mathbb{E}_\mu(1_A|\mathcal{D})d\mu.$$
	One standard fact is that
	$H_\mu(\alpha|\mathcal{D})$ increases with respect to $\alpha$ and decreases with respect to $\mathcal{D}$. Set $\mathcal{N}=\{\emptyset,X \}$ and define
	$$H_\mu(\alpha)=H_\mu(\alpha|\mathcal{N})=-\sum_{A\in \alpha}\mu(A)\log \mu(A).$$

	Let $(X,\mu, G)$ be an mps and $\mathcal{A}=\{g_i\}_{i\in \N}$ be a sequence in $G$. The sequence entropy of $\alpha$ with respect to $(X, \mu, G)$ along $\mathcal{A}$ is 
	$$h_{\mu}^\mathcal{A}(G,\alpha)=\limsup_{n\to\infty}\frac{1}{n}H_\mu( \vee_{i=1}^{n}g_i^{-1}\alpha).$$  The supremun sequence entropy of $\alpha$ with respect to $(X, \mu, G)$ is  
	$$h_{\mu}^\mathcal{*}(G,\alpha)=\sup_{\mathcal{A}}h_{\mu}^\mathcal{A}(G,\alpha),$$
	where the supremun is taken over all sequences of $G$.
	Define the  supremum of sequence entropy of $(X,\mu,G)$ by 
	$$h_{\mu}^*(G)=\sup_{\alpha\in\mathcal{P}_X^\mu}h_{\mu}^*(G,\alpha).$$
	The following result establishes the relation between measure-theoretic sequence entropy and Kronecker algebra, which was proved in \cite{LiuWangxu2024} using filters and in \cite{KH09} via a combinatorial method (see \cite{Huang_Sequenceentropypairs} and \cite{LY} for the abelian and amenable cases). We note that all the proofs above are essentially the same as the proof in \cite{Huang_Sequenceentropypairs}.
	\begin{thm}\label{thm-se-enytopy}
		Let $(X,\mu, G)$ be an mps, and $\alpha\in\mathcal{P}_X^\mu$. Then $$h^*_\mu(G,\alpha)=H_\mu(\alpha|\mathcal{K}_\mu(G)).$$
	\end{thm}
	
	A measurable partition $\{ A_1,\ldots,A_L\}\in\Pu$ is said to be  \emph{admissible} with respect to a $K$-tuple $(x_1,\ldots,x_K)\in X^{(K)}$ if for each $1\le l\le L$, there exists $1 \le k_l\le n$ such that $x_{k_l}\notin \overline{A_l}$, where  $\overline{A}$ is the closure of the subset $A$ in $X$.  Denote $\Delta_K(X):=\{(x,\ldots,x)\in X^{(K)}:x\in X\}$ for $K\ge 2$.
	We  say a tuple $(x_1,\dots, x_K) \in X^{(K)}\setminus \Delta_K(X)$ is a \emph{$\mu$-sequence entropy $K$-tuple} if, for any   admissible partition $\alpha\in\Pu$ with respect to $(x_1,\ldots,x_K)$, $h_\mu^*(G,\alpha)>0$. Denote by $SE_K^\mu(X, G)$ the set of all $\mu$-sequence entropy $K$-tuples.
	
 Let $\mu=\int_X\mu_xd\mu(x)$ be the disintegration of $\mu$ over $\mathcal{K}_\mu(G)$.
			For $K\in\mathbb{N}$ with  $K\ge 2$, define a measure on $X^{(K)}$ via
	\begin{align}\label{eq-1213}
		\lambda_K^X=\int_X\mu_x \times \ldots \times \mu_x \, d\mu(x).
	\end{align}
	
Now we establish the relation between measure-theoretic sequence entropy tuple and the topological support of $\l_K^X$, following the approach for $\mathbb{Z}$-actions  in \cite[Theorem 3.4]{Huang_Sequenceentropypairs}. Thus, the proof of Proposition \ref{suppordisin} is provided in  Appendix.
	\begin{prop}\label{suppordisin}Let $(X,G)$ be a tds and $\mu\in\mathcal{M}(X,G)$ For $K\ge 2$,
		$$
		SE_K^{\mu}(X, G) = \operatorname{supp}(\lambda_K^X )\setminus \Delta_K(X).
		$$
	\end{prop}

	\subsection{Mean sensitivity}
	If we further assume that $G$ is amenable, we can introduce, similarly to the case of $G=\mathbb{Z}$ (as in \cite{LYY22}), the notions of mean sensitive tuples and  sensitive in the mean tuples.  Let $\mathcal{F}=\{ F_n\}_{n\ge 1}$ be a F{\o}lner sequence of $G$. We  say that the $K$-tuple $(x_1,\dots, x_K) \in X^{(K)}\setminus \Delta_K(X)$ is
	\begin{itemize}
		\item[(1)]a \textbf{$\mu$-mean sensitive $K$-tuple along $\mathcal{F}$}  if for any open neighborhoods $U_k$ of $x_k$ with $k=1,\dots, K$, there exists $\delta > 0$ such that for any $A \in \mathcal{B}_X^\mu$ with $\mu(A) > 0$, there are $y_1, y_2, \dots, y_K\in A$ such that 
		$$\limsup_{n\to\infty}\frac{|\{ g\in F_n:gy_k\in U_k,k=1,\dots,K \}|}{|F_n|}>\delta;$$
		\item[(2)]a \textbf{$\mu$-sensitive in the mean $K$-tuple along $\mathcal{F}$}  if for any open neighborhoods $U_k$ of $x_k$ with $k=1,\dots, K$, there exists $\delta > 0$ such that for any $A \in \mathcal{B}_X^\mu$ with $\mu(A) > 0$, there exists $n\in \mathbb{N}$ and $y_1^n, y_2^n, \dots, y_K^n \in A$  such that
		$$
		\frac{|\{g \in F_n : g y_k^n \in U_k, k = 1, 2, \dots,K\}|}{|F_n|} > \delta.
		$$
	\end{itemize}
	We denote the sets of all $\mu$-mean sensitive $K$-tuples along $\mathcal{F}$ and  $\mu$-sensitive
	in the mean $K$-tuples along $\mathcal{F}$ by $\text{MS}_K^\mu(X,\mathcal{F})$ and  $\text{SM}_K^\mu(X,\mathcal{F})$, respectively.

	\section{Proof of Theorem \ref{main_abel}}\label{sec-it}
	In this section, we prove Theorems \ref{main_abel}, which follows as a corollary of Theorem \ref{thm:IT_support}. In the first subsection, we introduce the necessary notation, state Theorem \ref{thm:IT_support} and establish Theorems \ref{main_abel}, assuming Theorem \ref{thm:IT_support} holds. The next subsection is devoted to the proof of Theorem \ref{thm:IT_support}.

	\subsection{Statement of Theorem \ref{thm:IT_support} and derivation of Theorem  \ref{main_abel}  from  Theorem \ref{thm:IT_support}}
	Let $(X,G)$ be a tds and $\mu \in \mathcal{M}(X,G)$. Let $\mu = \int_X \mu_x \, d\mu(x)$ be the disintegration of $\mu$  over $\mathcal{K}_\mu(G)$.
	\begin{lem}\label{lem-11}For $n\in\N$, let $\mu^{(n)} = \int_{X^{(n)}} \rho_{\mathbf{x}} \, d\mu^{(n)}(\mathbf{x})$ be the ergodic decomposition of $\mu^{(n)}$. Then 		for $\mu^{(n)}$-a.e. $\mathbf{x} \in  X^{(n)}$,
		\begin{equation}\label{disinteg_Kron_n}
			\rho_{\mathbf{x}}^{(n)} = \int_{X^{(n)}} \mu_{x_1} \times \cdots \times \mu_{x_n} \, d\rho_{\mathbf{x}}^{(n)}(x_1,\ldots,x_n).
		\end{equation}
	
	\end{lem}
	\begin{proof}By Lemma \ref{n_kronecker}, one has $\mathcal{K}_{\mu^{(n)}}(G)= (\mathcal{K}_{\mu}(G))^{(n)}$. Therefore, Lemma \ref{lem-11} follows directly from \eqref{eq1} and the fact that  $\mathcal{I}_{\mu^{(n)}}(G)\subset\mathcal{K}_{\mu^{(n)}}(G)$.
	\end{proof}

	Define the space
	$$
	\mathcal{X}_\mu = \left\{ \rho \in \mathcal{M}(X) : \mu\left( \{ x : \varrho(\rho, \mu_x) < \epsilon \} \right) > 0 \ \text{ for any }  \epsilon > 0 \right\},
	$$
	where $\varrho$ is the metric on $\mathcal{M}(X)$ defined as in \eqref{metric_measure}. The following lemma shows that $\mathcal{X}_\mu$ contains almost all measures arising from the disintegration over the Kronecker algebra.
	\begin{lem}\label{lem-full} 
		Let $R_\mu := \{ x \in X : \mu_x \in \mathcal{X}_\mu \}$. Then $\mu(R_\mu) = 1$.
	\end{lem}
	\begin{proof}
		Since the map $x \mapsto \mu_{x}$ is measurable, Lusin's theorem ensures that for any $\epsilon > 0$, there exists a closed set $X_\epsilon \subseteq X$ with $\mu(X_\epsilon) > 1 - \epsilon$, on which $x \mapsto \mu_{x}$ is continuous and $X_\epsilon = \operatorname{supp}(\mu|_{X_\epsilon})$. It is easy to see that $X_\epsilon \subseteq R_\mu$. Since $\epsilon>0$ is arbitrary, it follows that $\mu(R_\mu) = 1$.
	\end{proof}
	
	We say that a subset $E$ of $X$ is an {\it IT set }if  each tuple of $E$ is an IT tuple.	Now we statement our main result in this section.
	\begin{thm}\label{thm:IT_support}
		For every  $\rho \in \mathcal{X}_\mu$,  $\mathrm{supp}(\rho)$ is an IT set.
	\end{thm}
We prove Theorem \ref{main_abel}, assuming Theorem \ref{thm:IT_support} holds.
	\begin{proof}[Proof of Theorem \ref{main_abel}]
	Given $(x_1, \dots, x_K)\in SE_K^\mu(X, G)$, we only need to prove that for any open neighborhood $U_1\times\cdots\times U_K$ of $(x_1,\ldots,x_K)$, $(U_1,\ldots, U_K)$ admits an infinite independence set.
		
		Now fix $(x_1, \dots, x_K)\in SE_K^\mu(X, G)$  and an  open neighborhood $U_1\times\cdots\times U_K$ of $(x_1,\ldots,x_K)$. 	By Proposition~\ref{suppordisin}, we have $\lambda_K^X( \prod_{k=1}^K U_k) > 0$, which implies
		$$
		\mu(\cap_{k=1}^K \{x \in X : \mu_x(U_k) > 0\}) > 0.
		$$
	 By Lemma~\ref{lem-full}, we have
		$$
		\mu(R_\mu \cap \cap_{k=1}^K \{x \in X : \mu_x(U_k) > 0\}) > 0.
		$$
		By taking $z\in R_\mu \cap \cap_{k=1}^K \{x \in X : \mu_x(U_k) > 0\}$, we have
		$$\text{supp}(\mu_{z})\cap U_k\neq\emptyset,\text{ for }k=1,\ldots,K\text{ and }\mu_{z}\in\mathcal{X}_\mu.$$
By Theorem \ref{thm:IT_support}, $\text{supp}(\mu_{z})$ is an IT set, and we can choose $x_k'\in \text{supp}(\mu_{z})\cap U_k$ for $k=1,\ldots,K.$ Thus, $(x_1',\cdots,x_K')$ is an  IT tuple, and so $(U_1,\ldots,U_K)$ admits an infinite independence set. By the arbitrariness of $U_1\times\cdots\times U_K$, the proof is complete.
	\end{proof}

\subsection{Proof of  Theorem~\ref{thm:IT_support}}
	\begin{proof}[Proof of Theorem~\ref{thm:IT_support}]
		
		Given  $\xi\in \mathcal{X}_\mu$ and $x_1,\ldots,x_K\in\text{supp}(\xi)$, we are going to show that $(x_1,\cdots,x_K)$ is an IT tuple, which implies that  $\text{supp}(\xi)$ is an IT set by the arbitrariness of $(x_1,\cdots,x_K)$. 
		
	Recall $$\mu = \int_X \mu_x \, d\mu(x)$$ is the disintegration of $\mu$ over $\mathcal{K}_\mu(G)$. Given open neighborhoods $U_k$ of $x_k$ for $k=1,\ldots,K$, 
	by Lemma \ref{lem-1} and Lemma \ref{lem-full}, the set
		\begin{align*}
			\tilde L:=\{x\in X: \mu_{ {x}}(U_k)>0,\text{ for }k=1,\ldots,K \}
		\end{align*}
		has positive measure.  Let $L$ be a compact subset of $\tilde L$ with $\mu(L)>0$. We can assume, w.l.o.g., that 
		\begin{align}\label{eq-22} \text{supp}(\mu|_L)=L.
		\end{align}
		\begin{cl}\label{cl-1}If open subsets $W_1,\ldots,W_M$ of $X$ satisfies $\mu(W_m\cap L)>0$ for $m=1,\ldots,M$, then for any finite subset  $F\subset G$,  there exist open  subsets $W_{m,k}$ of $X$ and $g_*\in G\setminus F$ such that  for $m=1,\ldots,M,k=1,\ldots,K$,
			\begin{itemize}
				\item[(1)] $\mu(W_{m,k}\cap L)>0$,
				\item[(2)]  $g_*W_{m,k}\subset U_k$,
				\item[(3)]$W_{m,k}\subset W_m$.
			\end{itemize}
		\end{cl}
Now, assuming Claim \ref{cl-1}, we complete the proof of Theorem \ref{thm:IT_support}.
			Put $W=X$. By Claim \ref{cl-1},  there exist open subsets $W_{1},\ldots,W_{K}$ of $X$ and $g_1\in G$ such that for $k=1,\ldots,K$,
		\begin{itemize}
			\item[(1)] $\mu(W_{k}\cap L)>0$;
			\item[(2)]  $g_1W_{k}\subset U_k$.
		\end{itemize}
Item (1) allows us to repeatedly apply Claim \ref{cl-1}. By induction, there exist  nonempty open sets $W_{a}\subset X$ for each $a\in \cup_{n=1}^\infty\{1,2,\ldots,K\}^{(n)}$ and a sequence $(g_n)_{n\in\mathbb{N}}$ of $G$ such that for  every $n\in\mathbb{N}$ and $\textbf{a}=a_1a_2\ldots a_n\in \{1,2,\ldots,K\}^{(n)}$,
		\begin{itemize}
			\item[(a)] $ W_{\textbf{a}k}\subset W_\textbf{a}$ for $k\in\{1,2,\ldots,K\}$;
			\item[(b)] $g_nW_{\textbf{a}}\subset U_{a_n}$;
			\item[(c)] $g_n\notin \{g_1,\ldots,g_{n-1}\}$.
		\end{itemize}
		Thus for any $N\in\N$ and $(a_n)_{n=1}^N\in \{1,2,\ldots,K\}^{(N)}$, one has 	$$\cap_{n=1}^N g_n^{-1}U_{a_n}\neq\emptyset.$$
		Therefore, $\{g_n\}_{n=1}^\infty$ is an infinite independence set of $(U_1,\ldots, U_K)$. By the argument presented at the start of the proof, the theorem is now established.
		\end{proof}
				
		It remains to prove Claim \ref{cl-1}. 
		\begin{proof}[Proof of Claim \ref{cl-1}] Rewrite  ${\bf x}=(x_{1,1},\cdots,x_{M,K})\in X^{MK}$, and $ \mu_{\bf x}=\mu_{ x_{1,1}}\times\cdots\times\mu_{x_{M,K}}$ for ${\bf x}\in X^{(MK)}$. Let  $$\mu^{(MK)}=\int_{X^{(MK)}} \rho_{\bf x}d\mu^{(MK)}(\bf x)$$	be the ergodic decomposition of $\mu^{(MK)}$. Denote $L_m=W_m\cap L$ for $m=1,2,\ldots,M$. By hypothesis, one has
			\begin{align}\label{eq-1}
				\mu^{(MK)}(L_1^{(K)}\times\cdots\times L_M^{(K)})>0.
			\end{align}
			Together with Lemma \ref{lem-11} and \eqref{eq-1}, we can find an ergodic measure $\rho$ of the product system $(X^{(MK)},G)$ such that 
		\begin{equation}\label{eq:3222143}
		\rho(L_1^{(K)}\times\cdots\times L_M^{(K)})>0,
		\end{equation}
			and
			\begin{equation}\label{eq:3222144}\rho=\int_{X^{(MK)}} \mu_{\bf x}d\rho(\bf x).\end{equation}
			Since $L_1^{(K)}\times\cdots\times L_M^{(K)}\subset \tilde L^{(MK)}$, for any ${\bf x}\in L_1^{(K)}\times\cdots\times L_M^{(K)}$,
			$$\mu_{\bf x}((U_1\times \cdots \times U_K)^{(M)})=\mu_{ x_{1,1}}(U_1)\times\cdots\times \mu_{ x_{1,K}}(U_K)\times\cdots\times \mu_{ x_{M,K}}(U_K)>0,$$
		which together with \eqref{eq:3222143} and \eqref{eq:3222144}, implies that
			\begin{align*}\rho((U_1\times \cdots \times U_K)^{(M)})>0.
			\end{align*}
			The ergodicity of $\rho$ ensures that  there exists $g_*\in G\setminus F$ such that 
			$$\rho((L_1^{(K)}\times\cdots\times L_M^{(K)})\cap g_*^{-1}((U_1\times \cdots \times U_K)^{(M)}))>0.$$
			Thus, $L_m\cap g_*^{-1}(U_k)\neq\emptyset$  for $m=1,\ldots,M, k=1,\ldots,K.$ Notice that every $g_*^{-1}(U_k)$ is an open subset of $X$, by \eqref{eq-22} and the assumption that $\mu(L_m)>0$,
			$$\mu(L_m\cap g_*^{-1}(U_k))>0\text{ for }m=1,\ldots,M, k=1,\ldots,K.$$
		For $m=1,\ldots,M,k=1,\ldots,K$, let
			$$W_{m,k}=W_m\cap  g_*^{-1}(U_k) $$
		Then for $m=1,\ldots,M,k=1,\ldots,K$,
			$$\mu(L\cap W_{m,k})= \mu(L_m\cap g_*^{-1}(U_k))>0.$$
		Thus, the open  subsets $W_{m,k}$,  $m=1,\ldots,M,k=1,\ldots,K$, and $g_*\in G\setminus F$ satisfy (1), (2) and (3). The proof of Claim \ref{cl-1} is completed.
	\end{proof}
	
	\section{Proof of Theorem \ref{thm-sequence-entropy}}\label{sec:new}
 Let $(X,G)$ be a tds and $\mu\in\mathcal{M}(X,G)$. Recall that  the disintegration of $\mu$ over $\mathcal{K}_\mu(G)$ is given by $\mu=\int_X\mu_xd\mu(x)$, and for	$K\in\N \text{ with }K\ge 2$,
\begin{align*}
	\lambda_K^X=\int_X\mu_x \times \cdots \times \mu_x \, d\mu(x).
\end{align*}
Before proving Theorem \ref{thm-sequence-entropy}, we establish the following lemma.
\begin{lem}\label{lem"sim. ob}  Let $K\in\N$ with  $K\ge 2$. If  $(X,G)$ has no essential sequence entropy $K$-tuples for $\mu$, then 
	$$|    \supp(\mu_x)|\le K-1,\text{ for $\mu$-a.e. }x\in X.$$                                                  
\end{lem}
\begin{proof}
Suppose for a contradiction that there exists a measurable subset $A$ of $X$ with $\mu(A)>0$ such that for all $x\in A$,
	$$|\text{supp}(\mu_x)|\ge K.$$  
	By Lusin's theorem, there exists a compact subset $F\subset A$ with $\mu(F)>0$ such that the map $x\mapsto \mu_x$ is continuous on $F$. Fix an essential tuple $(x_1,x_2,\ldots,x_K)\in \supp(\mu_z)^{(K)}$ for some $z\in F$. We now prove $(x_1,x_2,\ldots,x_K)$ is an essential sequence entropy $K$-tuples for $\mu$, yielding a contradiction that completes the proof.

	For any neighborhood $U_k$ of $x_k$, for $k=1,2,\ldots,K$, 
	\[(\mu_z\times \cdots\times\mu_z)(U_1\times\cdots\times U_K)>0.\]
	By the continuity of the map $x\mapsto \mu_x$ is continuous on $F$ and Lemma \ref{lem-1}, there exists an open neighborhood $U$ of $z$ such that $\mu(F\cap U)>0$ and for any $x'\in F\cap U$, 
	\[(\mu_{x'}\times \cdots\times\mu_{x'})(U_1\times\cdots\times U_K)>0.\]
	Thus, 
	\[	\lambda_K^X(U_1\times \cdots\times U_K)\ge \int_{F\cap U}(\mu_x\times \cdots\times\mu_x)(U_1\times\cdots\times U_K) \, d\mu(x)>0.\]
	By Proposition \ref{suppordisin}, one has $(x_1,x_2,\ldots,x_K)$ is an essential sequence entropy $K$-tuples for $\mu$. The proof is completed by the argument given at the beginning.
\end{proof}

We are now in a position to complete the proof of Theorem \ref{thm-sequence-entropy}.
	\begin{proof}[Proof of Theorem \ref{thm-sequence-entropy}]	Let $K\in\N$ with $K\ge 2$. We now suppose that  $(X,G)$ has   essential sequence entropy $K$-tuples for $\mu$, and prove $h_\mu^*(G)\le \log(K-1)$.
		
			By Lemma \ref{lem"sim. ob}, one has 	$$|    \supp(\mu_x)|\le K-1,\text{ for $\mu$-a.e. }x\in X.$$      
		Then for any $\alpha\in\mathcal{P}_X^\mu$,  
		$$|\{A\in\alpha: \mu_x(A)>0  \}|\le K-1,\text{ for $\mu$-a.e. $x\in X$}.$$
		Hence the convex of the function $-x\log x$ implies that
		\begin{align*}
			H_\mu(\alpha|\mathcal{K}_\mu)&=\int_X\sum_{A\in\alpha}-\mathbb{E}_\mu(1_A|\mathcal{K}_\mu)(x)\log \mathbb{E}_\mu(1_A|\mathcal{K}_\mu)(x)d\mu(x)\\
			&=\int_X\sum_{A\in\alpha}-\mu_x(A)\log \mu_x(A)d\mu(x)\\
			&=\int_X\sum_{A\in\alpha\atop \mu_x(A)>0}-\mu_x(A)\log \mu_x(A)d\mu(x)\\
			&\le \int_X-(\sum_{A\in\alpha\atop \mu_x(A)>0}\mu_x(A))\log \frac{\sum_{A\in\alpha\atop \mu_x(A)>0}\mu_x(A)}{|\{A\in\alpha: \mu_x(A)>0\}|}d\mu(x)\\
			&\le \int_X\log(K-1) d\mu(x)=\log (K-1).
		\end{align*}
		
		By Theorem \ref{thm-se-enytopy},  we have $h^{*}_{\mu}(G,\alpha)=H_\mu(\alpha|\mathcal{K}_\mu)\leq \log (K-1)$. With the arbitrariness of $\alpha$, we have  
		$h^{*}_{\mu}(G)\leq\log (K-1),$ which proves Theorem \ref{thm-sequence-entropy}.
	\end{proof}
	\section{Proof of Theorem \ref{cor:se=sm}}\label{sec:proof_thm_se_ms}
In the first subsection, we present the decomposition of ergodic measures via their Kronecker factor, which plays a key role in the proof of Theorem \ref{cor:se=sm}. In the following subsection, we complete the proof of Theorem \ref{cor:se=sm}. We note that the ergodic decomposition in the first subsection holds without assuming that the group 
$G$ is amenable.
	\subsection{Decomposition of ergodic measures}\label{subsec:ergodic} 
	Let $(X,G)$ be a tds, and let $\mu \in \mathcal{M}^e(X,G)$. Let $\pi \colon (X,\mu,G) \to (Y,\nu,G)$ denote the factor map to its Kronecker factor, with $\mu = \int_{Y} \mu_y \, d\nu(y)$ being the disintegration  of $\mu $ over $\pi$. 
	Proposition~\ref{vonNeumann} induces the  following diagram:
	\begin{align}\label{graph}\begin{split}
			\xymatrix{
				&~ &(X, \mu, G)  \ar[d]_{\pi}              \\
				&	(Z,\nu_Z, G) \ar[r]^{\tau } & (Z/H,\nu_{Z/H}, G), &}
		\end{split}
	\end{align}
	where $\pi:X\to Z/H$ is the factor map to the Kronecker factor of $(X, \mu, G)  $, and $\tau,Z,H,\nu_Z,\nu_{Z/H},\theta$ are as in Proposition \ref{vonNeumann}.  Recall that $\theta: G\to Z$ is a continuous group homomorphism. The denseness of $\theta(G)$ in $Z$ implies that $(Z,G)$ is both minimal and uniquely ergodic. For $n\in \N$ and ${\bf z}=(z_1,\dots,z_n)\in Z^{(n)}$, denote $$Z{\bf z}=\{ z{\bf z}:z\in Z\}=\{(zz_1,\dots,zz_n):z\in Z \}.$$
	It is easy to see that $(Z{\bf z},G)$ is a subsystem of $(Z^{(n)},G)$. The following lemma about the system $(Z{\bf z},G)$ is straightforward.
	\begin{lem}\label{lem-Zz}For $n\in\N$ and ${\bf z}\in Z^{(n)}$, $(Z{\bf z},G)$ is topologically conjugate to $(Z,G)$ with respect to the factor map  $p_{\bf z}: z\to z{\bf z}.$ In particular,  $(Z{\bf z},G)$ is both minimal and uniquely ergodic.
	\end{lem}
	Lemma \ref{lem-Zz} ensures that the unique ergodic measure $\lambda_{\bf z}^Z$ of $(Z{\bf z},G)$ is of the following form:
	\begin{align}\label{de-2-lambda-zz}
		\lambda_{\bf z}^Z=(p_{\bf z})_*(\nu_Z)=\int_{Z}\delta_{z{\bf z}}d\nu_Z(z),
	\end{align} 
	where $p_{\bf z}$ is defined in Lemma \ref{lem-Zz} and $\nu_Z$ is the Haar measure on $Z$.
	\begin{lem}\label{lem-er-decom-Z}For $n\in\N$, $\int_{Z^{(n)}} \lambda_{\bf z}^{Z}d\nu_Z^{(n)}({\bf z})$ is the ergodic decomposition of $\nu_{Z}^{(n)}$. %Namely, $\int_{Z^{(n)}} \lambda_{\bf z}^{Z}d\nu_Z^{(n)}({\bf z})$ is the disintegration of $\nu_Z^{(n)}({\bf z})$ with respect to the sub-$\sigma$-algebra consisting of all invariant sets.
	\end{lem}
	\begin{proof}By \eqref{de-2-lambda-zz}, one obtains that
		\begin{align*}
			\int_{Z^{(n)}} \lambda_{\bf z}^{Z}d\nu_Z^{(n)}({\bf z})&=\int_{Z^{(n)}}\int_{Z}\delta_{z\bf z}d\nu_Z(z)d\nu_Z^{(n)}({\bf z})\\
			&=\int_{Z}\int_{Z^{(n)}}\delta_{z\bf z}d\nu_Z^{(n)}({\bf z})d\nu_Z(z)\\
			&=\int_{Z}\nu_Z^{(n)}d\nu_Z(z)=\nu_Z^{(n)}.
		\end{align*}
		We finish the proof of Lemma \ref{lem-er-decom-Z}, since  $\lambda_{\bf z}^{Z}$ is ergodic for every ${\bf z}\in Z^{(n)}$.
	\end{proof}
	Using the quotient map $\tau: Z \to Z/H$ in Proposition \ref{vonNeumann}, for $n\in\N$ and ${\bf z}\in Z^{(n)}$, we can define an ergodic measure on $(Z/H)^{(n)}$ by
	\begin{align}\label{eq-de-2}
		\lambda_{\bf z}^{Z/H}:=(\tau^{(n)})_*(\lambda_{\bf z}^Z)=\lambda_{\bf z}^Z\circ(\tau^{-1})^{(n)}.
	\end{align}
	
	\begin{lem}\label{lem-er-decom-Z/H}For $n\in\N$, $\int_{Z^{(n)}} \lambda_{\bf z}^{Z/H}d\nu_Z^{(n)}({\bf z})$ is the ergodic decomposition of $\nu_{Z/H}^{(n)}$.
	\end{lem}
	\begin{proof}By Lemma \ref{lem-er-decom-Z} and \eqref{eq-de-2},
		\begin{align*}
			\int_{Z^{(n)}} \lambda_{\bf z}^{Z/H}d\nu_Z^{(n)}({\bf z})&=\int_{Z^{(n)}}(\tau^{(n)})_*(\lambda_{\bf z}^Z)d\nu_Z^{(n)}({\bf z})
			\\
			&=(\tau^{(n)})_*\left(\int_{Z^{(n)}}\lambda_{\bf z}^Zd\nu_Z^{(n)}({\bf z})\right)
			=(\tau^{(n)})_*(\nu_Z^{(n)})=\nu_{Z/H}^{(n)}.
		\end{align*}
		This  proves Lemma \ref{lem-er-decom-Z/H} since $\lambda_{\bf z}^{Z/H}$ is ergodic  for every ${\bf z}\in Z^{(n)}$.
	\end{proof}
	Recall that $\pi:(X,\mu,G)\to (Z/H,\nu_{Z/H},G)$  is the factor map to its Kronecker factor.  Let $\mu=\int_{Z/H} \mu_yd\nu_{Z/H}(y)$ be the disintegration of $\mu$ over $\pi$. For $n\in\N$ and ${\bf z}\in Z^{(n)}$, define
	\begin{align}\label{de-lambda-zx-1}\lambda_{\bf z}^{X}:=\int_{(Z/H)^{(n)}}\mu_{y_1}\times\ldots\times \mu_{y_n}d\lambda_{\bf z}^{Z/H}(y_1,\ldots,y_n).\end{align}
	Then, 
	\begin{align}\label{de-lambda-zx}
		\lambda_{\bf z}^{X}=\int_{Z^{(n)}}\mu_{\tau(z_1)}\times\ldots\times \mu_{\tau(z_n)}d\lambda_{\bf z}^{Z}(z_1,\ldots,z_n).
	\end{align}
	Here $\lambda_{\bf z}^{X}$ is well defined since the map $y\mapsto \mu_y$ is measurable.
	\begin{lem}\label{lem-er-demco-Xn} For $n\in\N$, $\int_{Z^{(n)}}\lambda_{\bf z}^{X}d\nu_Z^{(n)}({\bf z})$ is the ergodic decomposition of $\mu^{(n)}$. 
	\end{lem}
	\begin{proof} 
			For convenience, we denote ${\bf x}=(x_1,\dots,x_n)\in X^{(n)}$ and ${\bf y}=(y_1,\dots,y_n)\in (Z/H)^{(n)}$.
		By Lemma \ref{n_kronecker},
		\begin{align}\label{eq-123}\mathcal{I}_{\mu^{(n)}}(G)\subset \mathcal{K}_{\mu^{(n)}}(G)=(\mathcal{K}_{\mu}(G))^{(n)}=(\pi^{-1}(\mathcal{B}_{Z/H}^{\nu_{Z/H}}))^{(n)}.
		\end{align}
	 Let 
		$$\int_{X^{(n)}} \mu_{\mathcal{I}, {\bf x}}d\mu^{(n)}({\bf x})\text{ and }\int_{X^{(n)}} \mu_{\mathcal{K}, {\bf x}}d\mu^{(n)}({\bf x})$$
		be the disintegration of $\mu^{(n)}$ over $\mathcal{I}_{\mu^{(n)}}(G)$ and $\mathcal{K}_{\mu^{(n)}}(G)$, respectively. By \eqref{eq-123},  for $\mu^{(n)}$-a.e. ${\bf x}\in X^{(n)}$,
		$$\mu_{\mathcal{K},{\bf x}}=\mu_{\pi^{(n)}(\bf x)},$$
		where $\mu_{\pi^{(n)}(\bf x)}:=\mu_{\pi(x_1)}\times\dots\times\mu_{\pi(x_n)}$ and $\mu=\int_{Z/H}\mu_yd\nu_{Z/H}(y)$ is the disintegration of $\mu$ over $\pi$. Then by \eqref{eq1}, for $\mu^{(n)}$-a.e. ${\bf x}\in X^{(n)}$,
		\begin{align}\label{eq-12131}\begin{split}
				\mu_{\mathcal{I}, {\bf x}}&=\int_{X^{(n)}} \mu_{\mathcal{K},{\bf x}'}d \mu_{\mathcal{I}, {\bf x}}({\bf x}')=\int_{X^{(n)}} \mu_{\pi^{(n)}({\bf x}')}d \mu_{\mathcal{I}, {\bf x}}({\bf x}')\\
				&=\int_{(Z/H)^{(n)}} \mu_{\bf y}d (\pi^{(n)})_*(\mu_{\mathcal{I}, {\bf x}})({\bf y}).
			\end{split}
		\end{align} 
		Since $\mu_{\mathcal{I}, {\bf x}}$ is ergodic for $\mu^{(n)}$-a.e. ${\bf x} \in X^{(n)}$, it follows that $(\pi^{(n)})_*(\mu_{\mathcal{I}, {\bf x}})$ is ergodic for $\mu^{(n)}$-a.e. ${\bf x} \in X^{(n)}$. Moreover,
		\begin{align*}
			\int_{X^{(n)}}(\pi^{(n)})_*(\mu_{\mathcal{I}, {\bf x}})d\mu^{(n)}({\bf x})=(\pi^{(n)})_*\left(\int_{X^{(n)}}\mu_{\mathcal{I}, {\bf x}}d\mu^{(n)}({\bf x})\right)=(\pi^{(n)})_*(\mu^{(n)})=\nu_{Z/H}^{(n)}.
		\end{align*} 
		Therefore, 	$\int_{X^{(n)}}(\pi^{(n)})_*(\mu_{\mathcal{I}, {\bf x}})d\mu^{(n)}({\bf x})$ is the ergodic decomposition of $\nu_{Z/H}^{(n)}$, which coincides with  $\int_{Z^{(n)}} \lambda_{\bf z}^{Z/H}d\nu_Z^{(n)}({\bf z})$ due to the uniqueness of ergodic decomposition of $\nu_{Z/H}^{(n)}$ and Lemma \ref{lem-er-decom-Z/H}. Thus, to show that $$\lambda_{\bf z}^{X}=\int_{(Z/H)^{(n)}}\mu_{{\bf y}}d\lambda_{\bf z}^{Z/H}({\bf y})$$ is ergodic for  $\nu_Z^{(n)}$-a.e. ${\bf z}\in Z^{(n)}$, it is equivalent to show that  $$\int_{(Z/H)^{(n)}} \mu_{\bf y}d (\pi^{(n)})_*(\mu_{\mathcal{I}, {\bf x}})({\bf y})$$ is ergodic for  $\mu^{(n)}$-a.e. ${\bf x}\in X^{(n)}$. 
		The later is straightforward from  \eqref{eq-12131} and the fact $\mu_{\mathcal{I}, {\bf x}}$ is the ergodic for  $\mu^{(n)}$-a.e. ${\bf x}\in X^{(n)}$.
		
		Moreover, by \eqref{de-lambda-zx-1} and Lemma \ref{lem-er-decom-Z/H},
		\begin{align*}
			\int_{Z^{(n)}}\lambda_{\bf z}^{X}d\nu_Z^{(n)}({\bf z})&=	\int_{Z^{(n)}}\int_{(Z/H)^{(n)}}\mu_{y_1}\times\ldots\times \mu_{y_n}d\lambda_{\bf z}^{Z/H}(y_1,\ldots,y_n)d\nu_Z^{(n)}({\bf z})\\
			&=\int_{(Z/H)^{(n)}}\mu_{y_1}\times\ldots\times \mu_{y_n}d\left(\int_{Z^{(n)}}\lambda_{\bf z}^{Z/H}(y_1,\ldots,y_n)d\nu_Z^{(n)}({\bf z})\right)\\
			&=\int_{(Z/H)^{(n)}}\mu_{y_1}\times\ldots\times \mu_{y_n}d\nu_{Z/H}^{(n)}(y_1,\ldots,y_n)\\
			&=\mu^{(n)}.
		\end{align*}
		Thus, $\int_{Z^{(n)}}\lambda_{\bf z}^{X}d\nu_Z^{(n)}({\bf z})$ is the ergodic decomposition of $\mu^{(n)}$,  proving Lemma \ref{lem-er-demco-Xn}.
	\end{proof}

	\subsection{Proof of Theorem \ref{cor:se=sm}}\label{subsec:provetheoremsm-2} 	
	Suppose that $(X,G)$ is a tds, where $G$ is amenable, and  $\mu\in\mathcal{M}^e(X,G)$. 	For  $U\in\mathcal{B}_X^\mu$, we define
	$$Z_U=\{z\in Z:\mu_{\tau(z)}(U)>0\}.$$
	\begin{lem}\label{lem-le0}
		Let $n\in\N$, ${\bf z}=(z_1,\ldots,z_n)\in Z^{(n)}$, and  $U_1,\ldots,U_n\in\Bm$.
		Then,
		$\lambda_{\bf z}^{X}(U_1\times\ldots\times U_n)>0$ if and only if 
		$\nu_Z(\cap_{i=1}^{n}(Z_{U_i}z_i^{-1}))>0.$
	\end{lem}
	\begin{proof}By \eqref{de-lambda-zx}, $\lambda_{\bf z}^{X}(U_1\times\ldots\times U_n)>0$  if and only if 
		\begin{align*}\lambda_{\bf z}^{Z}(Z_{U_1}\times\ldots\times Z_{U_n})>0.
		\end{align*}
		By \eqref{de-2-lambda-zz}, it is also equivalent to that 
		$$\nu_Z(\cap_{i=1}^{n}(Z_{U_i}z_i^{-1}))=\int_{Z}\delta_{z{\bf z}}(Z_{U_1}\times\ldots\times Z_{U_n})d\nu_Z(z)>0,$$
		which proves Lemma \ref{lem-le0}.
	\end{proof}
	\begin{lem}\label{lem-le01}
		Let $n\in\N$, ${\bf z}=(z_1,\ldots,z_n)\in Z^{(n)}$, and let $U_1,\ldots,U_n\in\Bm$. If
		$\lambda_{\bf z}^{X}(U_1\times\ldots\times U_n)>2b$ for some $b>0$, then 
		$$\nu_Z(\cap_{i=1}^{n}(Z^b_{U_i}z_i^{-1}))>b,$$
		where $Z_U^b=\{z\in Z:\mu_{\tau(z)}(U)>b\}$ for  $U\in\Bm$.
	\end{lem}
	\begin{proof}
		By \eqref{de-lambda-zx}, 
		\begin{align*}2b&<\lambda_{\bf z}^{X}(U_1\times\ldots\times U_n)=\int_{Z^{(n)}}\mu_{\tau(z_1)}(U_1)\times\ldots\times \mu_{\tau(z_n)}(U_n)d\lambda_{\bf z}^{Z}(z_1,\ldots,z_n)\\
			&=\left(\int_{Z_{U_1}^b\times\ldots\times Z_{U_n}^b}+\int_{Z^{(n)}\setminus( Z_{U_1}^b\times\ldots\times Z_{U_n}^b)}\right)\mu_{\tau(z_1)}(U_1)\times\ldots\times \mu_{\tau(z_n)}(U_n)d\lambda_{\bf z}^{Z}(z_1,\ldots,z_n)\\
			&\le \lambda_{\bf z}^{Z}(Z_{U_1}^b\times\ldots\times Z_{U_n}^b)+b.
		\end{align*}
		Thus, 
		$$\nu_Z(\cap_{i=1}^{n}(Z^b_{U_i}z_i^{-1}))=\lambda_{\bf z}^{Z}(Z_{U_1}^b\times\ldots\times Z_{U_n}^b)>b,$$
		which proves Lemma \ref{lem-le01}.
	\end{proof}

	\begin{proof}[Proof of Theorem \ref{cor:se=sm}]Let $\mathcal{F}=\{ F_n\}_{n\in\N}$ be a tempered F{\o}lner sequence    of $G$. Since $\text{MS}_K^\mu(X,\mathcal{F})\subset\text{SM}_K^\mu(X,\mathcal{F})$ obviously holds, it is sufficient to show that $\text{SE}_K^\mu(X,G)\subset \text{MS}_K^\mu(X,\mathcal{F})$ and $\text{SM}_K^\mu(X,\mathcal{F})\subset\text{SE}_K^\mu(X,G)$.

		Firstly, we prove $\text{SE}_K^\mu(X,G)\subset \text{MS}_K^\mu(X,\mathcal{F})$.	
		Let $(x_1,\ldots,x_K)\in\text{SE}_K^\mu(X,G)$, and let $U_1\times\dots \times U_K$ be an open neighborhood of $(x_1,\ldots,x_K)$.  By Proposition \ref{suppordisin}, $(x_1,\ldots,x_K)\in\text{supp}(\lambda_K^X) $, and so there exists $b>0$ such that $$\lambda_K^X\left(U_1\times \ldots \times U_K\right)=2b> 0.$$ Given  $V\in\Bm$ with $\mu(V)>0$, one has  $\nu_Z(Z_V)>0$.	
		By Lemma \ref{lem-le01}, $\nu_Z(\cap_{k=1}^KZ^b_{U_k})>b$, which together with Lemma \ref{ll} and Lemma \ref{lem-er-demco-Xn}, implies that there exists ${\bf z}=(z_1,\ldots,z_K)\in Z^K$ such that 
		\begin{enumerate}
		\item $\lambda_{\bf z}^X$ is ergodic;
		\item $\nu_Z(\cap_{k=1}^K(Z^b_{U_k}z_k^{-1}))>b$;
		\item 	$\nu_Z(\cap_{k=1}^K(Z_{V}z_k^{-1}))>0$.
		\end{enumerate}
	 The item (2)  implies that
		\begin{align}\label{eq01}
			\lambda_{\bf z}^{X}(U_1\times\ldots\times U_K)&\ge \int_{Z_{U_1}^b\times\ldots\times Z_{U_K}^b}\mu_{\tau(z_1')}(U_1)\times\ldots\times \mu_{\tau(z_K')}(U_K)d\lambda_{\bf z}^{Z}(z_1',\ldots,z_K')\notag\\
			&\ge b^K \lambda_{\bf z}^{Z}(Z_{U_1}^b\times\ldots\times Z_{U_K}^b)\notag\\
			&=b^K\nu_Z(\cap_{k=1}^K(Z^b_{U_k}z_k^{-1}))\\
			&>b^{K+1}\notag.
		\end{align}
		By Lemma \ref{lem-le0}, 	 the item (3)  implies that
		\begin{equation}\label{eq1558}
			\lambda_{\bf z}^{X}(V^K)>0.
		\end{equation}
		By Theorem \ref{tempered_pointwise_ergodic} and \eqref{eq1558}, 	 the item (1)  implies that  there exist $x_1',\ldots,x_K'\in V$ such that 
		\begin{align*}
			\limsup_{n\to\infty}\frac{1}{|F_n|}\sum_{g\in F_n}1_{U_1\times\ldots\times U_K}(gx_1',\ldots,gx_K')
			=\lambda_{\bf z}^{X}(U_1\times\ldots\times U_K)>b^{K+1}.
		\end{align*}
		Thus, $(x_1,\ldots,x_K)\in  MS_K^{\mu}(X,\mathcal{F})$.

		Now following ideas in \cite{LLTS2024},	we prove $\text{SM}_K^\mu(X,\mathcal{F})\subset\text{SE}_K^\mu(X,G).$ Suppose for the contradiction, that there exists  $(x_1,\ldots,x_K)\in \text{SM}_K^\mu(X,\mathcal{F})$ with  $(x_1,\ldots,x_K)\notin SE_K^{\mu}(X,G)$. There exists an admissible partition  $\alpha=\{A_1,\ldots,A_L\}$ of $(x_1,\ldots,x_K)$ such that  $h^*_\mu(G,\alpha)=0$. Put $E_k=\{1\le l\le L: x_k\not\in \overline{A_l}\}$ for $1\le k\le K$. By the definition of admissible partition, $\cup_{k=1}^K E_k=\{1,\ldots,L\}$. Set
		$$B_1=\cup_{k\in E_1}A_k, B_2=\cup_{k\in E_2\setminus E_1}A_k, \ldots,   B_K=\cup_{k\in E_K\setminus(\cup_{j=1}^{K-1}E_j)}A_k.
		$$
		Then $\beta=\{B_1,\ldots,B_K\}$ is also an admissible partition of $(x_1,\ldots,x_K)$ such that $x_k\notin \overline{B_k}$ for all $1\le k\le K$. Without loss of generality, we assume $B_k\neq \emptyset$ for  $1\le k\le K$. Since $\alpha$ is finer than $\beta$ and  $h^*_\mu(G,\alpha)=0$, one has $h^*_\mu(G,\beta)=0$. 
		By Theorem \ref{thm-se-enytopy},  we have $\beta\subseteq \mathcal{K}_\mu(G)$.
		
		Take $\epsilon>0$ such that $B_\epsilon(x_k)\cap B_k=\emptyset$ for   $1\le k\le K$. Since $(x_1,\ldots,x_K)\in SM_K^{\mu}(X,G)$,  there is $\delta:=\delta(\epsilon)>0$   such that for any $A\in \mathcal{B}_X^\mu$ with $\mu(A)>0$ there are $m\in\mathbb{N}$ and $y_1^m,\ldots,y_K^m\in A$ such that  $|C_m| \ge m\delta$, where  $C_m=\{g\in F_m:gy_k^m\in B_\epsilon(x_k)\text{ for all }k=1,2,\ldots,K\}$.
		
		Since $ B_\epsilon(x_1)\cap B_1=\emptyset$, it follows that $B_\epsilon(x_1)\subset \cup_{k=2}^KB_k$, which implies that there is $k_0\in \{2,\ldots,K\}$ such that
		$$
		|\{g\in C_m: gy_1^m\in B_{k_0} \}|\ge \frac{|C_m|}{K-1}.
		$$
		For any $g\in C_m$, we have $gy_{k_0}^m\in B_\epsilon(x_{k_0})$,  and then  $gy_{k_0}^m\notin B_{k_0}$, as  $B_\epsilon(x_{k_0})\cap B_{k_0}=\emptyset$. This implies that
		\begin{equation}\label{e1}
			\frac{1}{|F_m|}\sum_{g\in F_m}|1_{B_{k_0}}(gy_1^m)-1_{B_{k_0}}(gy_{k_0}^m)|\ge\frac{|C_m|}{m(K-1)}\ge \frac{\delta}{K-1}.
		\end{equation}
		
		On the other hand, as for each $1\le k\le K$, $1_{B_k}$ is an almost periodic function, by \cite[Lemma 2.5]{YZZ}, for  any $\tau>0$, there is a compact subset $L\subset X$ with $\mu(L)>1-\tau$ such that for any $\epsilon'>0$, there is $\delta'>0$  such that for all $m\in\mathbb{N}$, whenever   $x,x'\in L$ with $d(x,x')<\delta'$,
		\begin{equation}\label{3}
			\frac{1}{|F_m|}\sum_{g\in F_m}|1_{B_k}(gx)- 1_{B_k}(gx')|<\epsilon'.
		\end{equation}
		Let $\epsilon'=\frac{\delta}{2(K-1)}$ and choose $\delta'>0$ such that it satisfies \eqref{3} with respect to $\epsilon'.$	Given a measurable subset $A\subset L$ such that $\mu(A)>0$ and $\text{diam}(A):=\sup\{d(x,x'):x,x'\in A\}<\delta'$. Then for any $m\in\mathbb{N}$ and $x,x'\in A$,
		$$
		\frac{1}{|F_m|}\sum_{g\in F_m}|1_{B_{k_0}}(gx)- 1_{B_{k_0}}(gx')|<\frac{\delta}{2(K-1)},
		$$
		which is a contradiction with \eqref{e1}.
		Thus, $SM_K^{\mu}(X,G)\subset SE_K^{\mu}(X,T)$.
	\end{proof}
	
	\appendix
	\section{Proof of Proposition \ref{suppordisin}}
	In this section, we prove we prove Proposition \ref{suppordisin}. To do this, we need some lemmas.
	\begin{lem}\label{lem-admis}Let $(X,G)$ be a tds with $\mu\in\mathcal{M}(X,G)$, and  $\alpha=\{ A_1,\ldots,A_L\}\in\Pu$. Suppose $\mu=\int \mu_xd\mu(x)$ is the disintegration of $\mu$ over $\mathcal{K}_\mu(G)$. Denote $C_l:=\{x\in X: \mu_{x}(A_l)>0\}$ for $1\le l\le L$. Then, $h^*_\mu(G, \alpha)=0$ if and only if $\mu(C_l\cap C_{l'})=0$ for every $1\le l<l'\le L$.
	\end{lem}
	\begin{proof}Assume that there exist $1\le l<l'\le L$ such that $\mu(C_l\cap C_{l'})>0$. Then 
		for $x\in C_l\cap C_{l'}$, $\mu_x(A_l)>0$ and $\mu_x(A_{l'})>0$, which implies that $\mu_x(A_l)\in(0,1)$, i.e., $A_l\notin \mathcal{K}_\mu(G)$. Hence, by Theorem \ref{thm-se-enytopy}, $h_\mu^*(G,\alpha)>0$.
		
		Conversely, assume $\mu(C_l \cap C_{l'}) = 0$ for all $1 \leq l < l' \leq L$. 
		Since $C_{l} \in \mathcal{K}_\mu(G)$, the integral decomposition gives:
		$$
		\mu(C_{l} \cap A_{l'}) = \int_{C_{l}} \mu_x(A_{l'}) \, d\mu(x) = 
		\begin{cases} 
			0, & l' \neq l \\
			\mu(A_{l'}), & l' = l
		\end{cases}
		$$
		This implies $\mu(C_{l} \triangle A_{l}) = 0$. 
		Consequently, $\alpha \subset \mathcal{K}_\mu(G)$, and thus $h^*_\mu(G, \alpha) = 0$ by Theorem \ref{thm-se-enytopy}.
	\end{proof}
	Now  we prove Proposition \ref{suppordisin}.
	\begin{proof}[Proof of Proposition \ref{suppordisin}]Firstly we show that $\operatorname{supp}(\lambda_K^X )\setminus \Delta_K(X)\subset  SE_K^{\mu}(X, G)$. Assume to the contrary that there exist $(x_1,\ldots,x_K)\in \operatorname{supp}(\lambda_K^X )\setminus \Delta_K(X)$, and    an admissible partition $\alpha=\{ A_1,\ldots,A_L\}$ with respect to $(x_1,\ldots,x_K)$, such that $h^*_\mu(G,\alpha)=0$.   By the argument similar to that in the proof of Theorem \ref{cor:se=sm}, we can find an admissible partition  $\beta=\{B_1,\ldots,B_K\}$ with respect to  $(x_1,\ldots,x_K)$ such that  $h^*_\mu(G,\beta)=0$ and $x_k\notin \overline{B_k}$ for all $1\le k\le K$. Without loss of generality, we assume $B_k\neq \emptyset$ for  $1\le k\le K$.
		
		Set $C_k=\{x\in X: \mu_{x}(B_k)>0\}$ for $1\le k\le K$. Since  $h^*_\mu(G,\beta)=0$, by Lemma \ref{lem-admis}, 
		$$\mu(C_k\cap C_{k'})=0\text{ for }1\le k<k'\le K.$$
		Thus, for $s\ge 2$ and $1\le k_1<k_2<\ldots<k_s\le K$,
		\begin{align*}
			\int_{X^{(K)}} \Pi_{t=1}^s1_{B_{k_t}}(x_{k_t})d \lambda_K^X(x_1,\ldots,x_K)=0,
		\end{align*}
		which shows that
		\begin{align*}
			&\int_{X^{(K)}} \Pi_{k=1}^K1_{X\setminus B_k}(x_k)d \lambda_K^X(x_1,\ldots,x_K)\\
			&=1-\sum_{k=1}^K\int_{X^{(K)}} 1_{B_k}(x_k)d \lambda_K^X(x_1,\ldots,x_K)\\
			&=1-\sum_{k=1}^K\mu(B_k)=0.
		\end{align*}
		So $\lambda_K^X(\Pi_{k=1}^K(X\setminus \overline{B_k}))=0$. This is impossible, since $(X\setminus \overline{B_1})\times \ldots\times (X\setminus \overline{B_K})$ is an open neighborhood of  $(x_1,\ldots,x_K)\in\operatorname{supp}(\lambda_K^X )$. Thus, $\operatorname{supp}(\lambda_K^X )\setminus \Delta_K(X)\subset  SE_K^{\mu}(X, G)$. 
		
		Secondly we show that $\text{SE}_K^\mu(X,G)\subset\operatorname{supp}(\lambda_K^X )\setminus \Delta_K(X)$. We only need to prove if  $(x_1,\ldots,x_K)\notin \operatorname{supp}(\lambda_K^X )\setminus \Delta_K(X)$, then $(x_1,\ldots,x_K)\notin SE_K^{\mu}(X, G)$. Now fix $(x_1,\ldots,x_K)\notin \operatorname{supp}(\lambda_K^X )\setminus \Delta_K(X)$. It is clear true in the case  $(x_1,\ldots,x_K)\in \Delta_K(X)$. So we assume that  $(x_1,\ldots,x_K)\notin \operatorname{supp}(\lambda_K^X )$. We will construct an admissible partition $\alpha$ with respect to $(x_1,\ldots,x_K)$ such that  $h^*_\mu(G,\alpha)=0$. Since   $(x_1,\ldots,x_K)\notin \operatorname{supp}(\lambda_K^X )$, there exists an open neighborhood $U_1\times\ldots\times U_K$ of $(x_1,\ldots,x_K)$ such that $\lambda_K^X(U_1\times\ldots\times U_K )=0$. Let $\mu=\int_X\mu_xd\mu(x)$ be the disintegration of $\mu$ over $\mathcal{K}_\mu(G)$.  Let $C_k=\{x\in X: \mu_{x}(U_k)>0\}\in\mathcal{K}_{\mu}(G)$ for $k=1,\ldots,K$. Then, $D_k:=U_k\cup C_k\in\mathcal{K}_{\mu}(G)$, as
		\begin{align*}
			\mu( U_k\setminus C_k)
			=\int_{X}\mu_{x}(U_k\cap C_k^c)d \mu(x)
			=0.
		\end{align*}
		Since  $\lambda_K^X(U_1\times\ldots\times U_K )=0$, one has 
		\begin{align}\label{eq-121311}\mu( \cap_{k=1}^K D_k)=\mu( \cap_{k=1}^K C_k)=0.\end{align}
		For any $\textbf{s}=(s(1),\ldots,s(K))\in \{0,1\}^K$, let $D_{\textbf{s}}=\cap_{k=1}^KD_k\left(s(k)\right)$, where $D_k(0)=D_k$ and $D_k(1)=D_k^c$. 
		Consider 
		$$\alpha=\left\{D_\textbf{s}:\textbf{s}\in\{0,1\}^K\setminus\{(0,\ldots,0)\}\right\}.$$
		By \eqref{eq-121311}, 
		$$\sum_{A\in\alpha}\mu(A)=1-\mu(D_{(0,\ldots,0)})=1-\mu( \cap_{k=1}^K D_k)=1.$$
		Hence, $\alpha\in \Pu$. For any $\textbf{s}=(s(1),\ldots,s(K))\in \{0,1\}^K\setminus\{(0,\ldots,0)\}$, we have $s(k)=1$ for some $k=1,\ldots,K$. Then $D_\textbf{s}\subset D_k^c\subset U_k^c$, implying that $x_k\notin \overline{D_\textbf{s}}$. Hence, $\alpha$ is  admissible   with respect to $(x_1,\ldots,x_K)$. Since $D_k\in\mathcal{K}_\mu(G)$, for $1\le k\le K$, one has $\alpha\subset \mathcal{K}_\mu(G)$, and then  by Theorem \ref{thm-se-enytopy}, $h^*_\mu(G,\alpha)=0$. Therefore, $(x_1,\ldots,x_K)\notin SE_K^{\mu}(X, G)$. Now we finish the proof.
	\end{proof}
	
	\noindent{\bf Acknowledgment.} 
	The authors thank Professor Garc\'{\i}a-Ramos for his many valuable suggestions on earlier versions of this paper.
This paper is  supported by National Key R\&D Program of China (No. 2024YFA1013602, 2024YFA1013600) and NNSF of China (12031019, 12371197, 12426201). 
	
\end{document}